\documentclass[11pt,reqno]{amsart}
\usepackage{amssymb,latexsym}
\usepackage{graphicx}
\usepackage{fancyhdr}

\usepackage[all]{xy}
\usepackage{enumerate}
\theoremstyle{plain}
\numberwithin{equation}{section}
\newtheorem{thm}{Theorem}[section]
\newtheorem{theorem}[thm]{Theorem}

\newtheorem{corollary}[thm]{Corollary}

\newtheorem{lemma}[thm]{Lemma}

\newtheorem{definition}[thm]{Definition}

\newtheorem*{theorem*}{Theorem}

\begin{document}

\setcounter{page}{1}

\title[On Isoclinism \& and Baer's theorem for Lie superalgebras]{On Isoclinism \&  Baer's theorem for Lie superalgebras}
\author[Nayak]{Saudamini Nayak}
\address{Department of Mathematics\\
National Institute of Technology Calicut\\
         NIT Campus P.O-673601, 
          Kozhikode,  India}
\email{ saudamini@nitc.ac.in}


\begin{abstract}\label{abstract}
In this paper we define isoclinism for Lie superalgebras and using the concept of isoclinism, we give the structure of all covers of Lie superalgebras when their Schur multipliers are finite dimensional. It has been shown that that maximal stem extensions of Lie superalgebras are precisely same as the stem covers. Furthermore, we have defined stem Lie superalgebra and prove that each isoclinic family $\mathcal{C}$ contain a stem Lie superalgebra $T$ and it is the one having minimum even and odd dimension. Finally we have proved a converse of Schur's theorem and  have given  bound for stem Lie superalgebra. Then we state and prove Baer's theorem( generalisation of Schur's theorem) and a converse of it.
\end{abstract}
\maketitle

\section{Introduction}\label{intro}
For a given group $G$, the notion of the Schur multiplier $\mathcal{M}(G)$ arose from the work of I. Schur on projective representation of groups as the second cohomology group with coefficients in $\mathbb{C}^{*}$ (see \cite{Schur1904}). Let $0\longrightarrow R \longrightarrow F \longrightarrow G \longrightarrow 0$ be a free presentation of the group $G$, then it can be shown by Hopf's formula that $\mathcal{M}(G) \cong R\cap [F,F]/[R,F]$ (see \cite{Kar1987}). 
The finite dimensional Lie algebra analogue to the Schur multiplier was developed in \cite{Batten1993} and later it has been studied by several authors. Let $A$ be a Lie algebra over a field $\mathbb{F}$ with a free presentation $0\longrightarrow R \longrightarrow F \longrightarrow A \longrightarrow 0$, where $F$ is a free Lie algebra. Then the Schur multiplier of $A$, denoted by $\mathcal{M}(A)$, and is defined to be the factor Lie algebra $(R \cap[F,F])/[R,F]$ \cite{Batten1993}. For  more information about the Schur multiplier of Lie algebras one can refer \cite{Batten1993,BMS1996,BS1996} and the references therein. Batten in \cite{Batten1993} showed that if $A$ is finite dimensional, then its Schur multiplier is isomorphic to $H^{2}(A,\mathbb{F})$, the second cohomology group of $A$. Similarly Schur multiplier for $n$-Lie algebras are defined and studied.  Schur multipliers are important because sometimes by just looking at the dimension of it, one can completely know the structure of corresponding Lie algebras (see \cite{BMS1996, Nir2011}). In 1940, P. Hall  introduced an equivalence relation on the class of all groups called isoclinism, which is weaker than isomorphism and plays an important role in classification of finite $p$-groups. In 1994, K. Moneyhun \cite{Moneyhun1994} gave a Lie algebra analogue to the concept of isoclinism.

\smallskip

The theory of Lie superalgebras and Lie supergroups has many applications in various areas of Mathematics and Physics. In 1975, Kac \cite{KAC1977} offered a comprehensive description of the mathematical theory of Lie superalgebras, and has established the classification of all finite-dimensional simple Lie superalgebras over an algebraically closed field of characteristic zero. Also he studied representation theory of classical simple Lie superalgebras and many results are extension of results of Lie algebras. So here we plan to extend the notation of isoclinism, Schur multiplier, covers of Lie algebras to Lie superalgebras and study some properties.

\smallskip

First we recall some definitions of Lie superalgebra, and fix some notation which we shall be using through out. Write $\mathbb{Z}_{2}=\{\bar{0}, \bar{1}\}$ is a field. A $\mathbb{Z}_{2}$-graded vector space is simply a direct sum of vector spaces $V_{\bar{0}}$ and $V_{\bar{1}}$ such that $V = V_{\bar{0}} \oplus V_{\bar{1}}$. A $\mathbb{Z}_2$-graded vector space (resp. $\mathbb{Z}_2$-graded algebra) is also referred to as a superspace (resp. superalgebra). We consider all vector superspaces and superalgebras are over $\mathbb{F}$(characteristic of $F \neq 2,3$). Elements in $V_{\bar{0}}$ (resp. $V_{\bar{1}}$) are called even (resp. odd) elements. Non-zero elements of $V_{\bar{0}} \cup V_{\bar{1}}$ are called homogeneous elements. For a homogeneous element $v \in V_{\sigma}$, with $\sigma \in \mathbb{Z}_{2}$ we set $|v| = \sigma$ is the degree of $v$. A  vector subspace $U$ of $V$ is called  sub superspace(or, subspace) if $U= (V_{\bar{0}} \cap U) \oplus (V_{\bar{1}} \cap U)$. We adopt the convention that whenever the degree function appeared in a formula, the corresponding elements are supposed to be homogeneous.
\smallskip

A {\it Lie superalgebra} \cite{Musson2012} is a superspace $L = L_{\bar{0}} \oplus L_{\bar{1}}$ with a bilinear mapping
$ [., .] : L \times L \rightarrow L$ satisfying the following identities:

\begin{enumerate}
\item $[L_{\alpha}, L_{\beta}] \subset L_{\alpha+\beta}$, for $\alpha, \beta \in \mathbb{Z}_{2}$ ($\mathbb{Z}_{2}$-grading),
\item $[x, y] = -(-1)^{|x||y|} [y, x]$ (graded skew-symmetry),
\item $(-1)^{|x||z|} [x,[y, z]] + (-1)^{ |y| |x|} [y, [z, x]] + (-1)^{|z| |y|} [z,[ x, y]] = 0$ (graded Jacobi identity).
\end{enumerate}
for all homogeneous elements $x, y, z \in L$. 

\smallskip

For a Lie superalgebra $L = L_{\bar{0}} \oplus L_{\bar{1}}$, the even part $L_{\bar{0}}$ is a Lie algebra and $L_{\bar{1}}$ is a $L_{\bar{0}}$-module. If $L_{\bar{1}} = 0$, then $L$ is just Lie algebra. Lie superalgebra without even part, i.e., $L_{\bar{0}} = 0$, is an abelian Lie superalgebra, as $[x, y] = 0$ for all $x, y \in L$. A sub superlgebra (or, subalgebra) of $L$ is a subspace which is closed under bracket operation. Define $[L, L] $ is the all finite linear combination of elements of the form $[x, y]$, is a subalgebra of $L$ is called derived subalgebra.  A graded ideal is a $\mathbb{Z}_{2}$-graded subspace $I$ of $L$ satisfying $[I,L]\subseteq I$. If $I$ and $J$ are graded ideals of $L$, then so is $[I, J]$. The derived subalgebra $L^{2}$ is a graded ideal. For
$Z(L) = \{z\in L : [z, x] = 0\;\mbox{for all}\;x\in L\}$
is a graded ideal and it is called the {\it center} of $L$. 
\smallskip

By a {\it homomorphism} between superspaces $f: V \rightarrow W $ of degree $|f|\in \mathbb{Z}_{2}$, we mean a linear map satisfying $f(V_{\alpha})\subseteq W_{\alpha+|f|}$ for $\alpha \in \mathbb{Z}_{2}$. In particular, if $|f| = \bar{0}$, then the homomorphism $f$ is called homogeneous linear map of even degree(or, degree zero map). A Lie superalgebra homomorphism $f: L \rightarrow M$ is a  homogeneous linear map of even degree such that $f[x,y] = [f(x), f(y)]$ holds for all $x, y \in L$. If $I$ is an ideal of $L$, the quotient Lie superalgebra $L/I$ inherits a canonical Lie superalgebra structure such that the natural projection map becomes a homomorphism. The notion of {\it epimorphisms, isomorphisms} and {\it automorphisms} have the obvious meaning. Throughout for superdimension of Lie superalgebra $L$ we simply write $\dim(L)=(m\mid n)$, where $\dim L_{\bar{0}} = m$, and $\dim L_{\bar{1}} = n$. 

\smallskip

For Lie superalgebra $L$ the lower central series is the descending series
\[ L=L^{1}  \supset L^{2}  \supset  \cdots \supset L^{c} \cdots\] where $L^{c}=[L^{c-1}, L]$ for all $ c \geq 1$. If $L^{c} = \{0\}$ for some $c$,  then $L$ is called nilpotent Lie superalgebra. So $L^{2}=[L,L]$, and $L$ is abelian if and only if $L^{2}=0$. Abelian Lie superalgebras are certainly nilpotent.
The upper central series of Lie superalgebra $L$ is the ascending series 
\[\{0\}=Z_{0}(L)\subseteq Z_{1}(L)\subseteq Z_{2}(L) \cdots\] given by the fact that  $Z_{1}(L) = Z(L)$, and  the $i$-th centre of $L$ is defined inductively by 
\[\frac{Z_i(L)}{Z_{i-1}(L)} = Z\left(\frac{L}{Z_{i-1}(L)}\right)\]
for all $i \geq 1$. Note here that $ Z_{i}(L)= \{ x \in L \mid [x, y] \in Z_{i-1}(L), \forall 
y \in L \}$ for $i \geq 1.$ We define $[x_{1}, x_{2}, \cdots, x_{t}]:=[\cdots, [[x_{1}, x_{2}], x_{3}] \cdots x_{t}]$ an element of $L^{t}$ where $x_{i} \in L$ for $1 \leq i \leq t$. With this notation, $ Z_{i}(L)= \{ x \in L \mid [x, y_{1}, \cdots, y_{i}]=0,  \forall y_{j} \in L \}$.

\smallskip

The free Lie superalgebra on a $\mathbb{Z}_{2}$-graded set $X = X_{\bar{0}} \cup X_{\bar{1}}$ is a Lie superalgebra $F(X)$ together with a degree zero map $i: X \rightarrow F(X)$ such that if $M$ is any Lie superalgebra, and $ j: X \rightarrow M$ is a degree zero map, then there is a unique Lie superalgebra homomorphism $h: F(X) \rightarrow M$ such that $ j = h \circ i$. The existence of free Lie superalgebra is guaranteed by an analogue of Witt's theorem (see \cite{Musson2012}).
\smallskip

So, if $L$ is a Lie superalgebra generated by a $\mathbb{Z}_{2}$-graded set $X = X_{\bar{0}} \cup X_{\bar{1}}$  and $\phi : X \rightarrow L$ is a degree zero map, then we have free Lie superalgebra $F$ and  $\psi: F \rightarrow L$ extending $\phi$. Let $R = \ker (\psi)$. The extension 
\[0 \longrightarrow R \longrightarrow F \longrightarrow L \longrightarrow 0\] 
is called a {\it free presentation} of $L$ and is denoted by $(F, \psi)$.
\begin{definition}
Given a free presentation,
$$ 0 \longrightarrow R \longrightarrow F \longrightarrow L \longrightarrow 0 $$
 of $L$ with $F$ a free Lie superalgebra, we define Schur multiplier of $L$ as  
$$ 
\mathcal{M}(L) = \frac{[F,F]\cap R}{[F, R]}.
$$ 
\end{definition}
 It is easy to see that the Schur multiplier of a Lie superalgebra $L$ is abelian, and independent of the choice of the free presentation of $L$.

Extensions of Lie superalgebras are studied by several authors \cite{Iohara2001, Neher2003}. An extension of a Lie superalgebra $L$ is a short exact sequence 
\begin{equation}\label{eq1}
\xymatrix{
0 \ar[r] & M\ar[r]^{e} & K\ar[r]^{f} & L \ar[r] & 0}. 
\end{equation}
Since $e: M \longrightarrow e(M) = \ker(f)$ is an isomorphism we will usually identify $M$ and $e(M)$. An extension of $L$ is then same as an epimorphism $f: K \longrightarrow L$. A homomorphism from an extension $f: K \longrightarrow L$ to another extension $\tilde{f}: \tilde{K} \longrightarrow L$ is a Lie superalgebra homomorphism $g: K \longrightarrow K'$ satisfying $f = \tilde{f} \circ g$; in other words, we have the following commutative diagram. 
\begin{center}
$\xymatrix{
K \ar[rd]_{f} \ar[rr]^{g} & & \tilde{K} \ar[ld]^{\tilde{f}} \\
& L &} $ 
\end{center}
A {\it central extension} of $L$ is an extension \eqref{eq1} such that $M=\ker f \subseteq Z(K)$. The central extension is said to be a {\it stem extension} of $L$ if $M \subseteq Z(K)\cap K^{2} $. 
\begin{definition} A  stem extension \eqref{eq1} is  called {\it maximal} if every epimorphism of any other stem extension of $L$ on to $0 \longrightarrow M \longrightarrow  K \longrightarrow   L \longrightarrow 0$ is necessarily an  isomorphism. 
\end{definition}
\begin{definition} Consider the  extension \eqref{eq1} is stem extension. It  is called a {\it stem cover} if $M \cong \mathcal{M}(L)$ and in this case $K$ is said to be a cover of Lie superalgebra $L$.
\end{definition}

 Now we introduce the notation of isoclinism for Lie superalgebras as follows.
\begin{definition}
Let $L$ and $K$ be two Lie superalgebras, $\alpha: \frac{L}{Z(L)}\longrightarrow \frac{K}{Z(K)}$ and $\beta: L^{2} \longrightarrow K^{2}$ be Lie superalgebra homomorphisms such that the following diagram is commutative
\begin{center}
$\xymatrix{
\frac{L}{Z(L)} \times \frac{L}{Z(L)} \ar[d]^{\alpha\times \alpha} \ar[r]^{\quad\quad\phi} &L^{2}\ar[d]^{\beta}\\
\frac{K}{Z(K)} \times \frac{K}{Z(K)} \ar[r]^{\quad\quad\psi}          &K^{2}}$
\end{center}
where $\phi: (\bar{l}, \bar{m}) \longrightarrow [l, m]$ for $ l, m \in L $  and similarly for $\psi: (\bar{r}, \bar{s}) \longrightarrow [r,s]$ for $r, s \in K$. Or, equivalently $\alpha$ and $\beta$ are defined in such a way that they are compatible, i.e., $\beta([l, m]) = [k, r]$, where $l, k, m, r \in L$ in which $k \in \alpha(\bar{l})$ and $ r \in \alpha(\bar{m})$.
Then the pair $(\alpha, \beta)$ is called {\it homoclinism} and if they are both isomorphisms, then $(\alpha, \beta)$ is called {\it isoclinism}.
\end{definition}
 If $(\alpha, \beta)$ is an isoclinism between $L$ and $K$, then $L$ and $K$ are said to be isoclinic, which we denote as  $L\sim K$. Isoclinism is an equivalence relation, and hence it breaks set of all Lie superalgebras into disjoint equivalence classes called isoclinism classes(or, isoclinism families). We show that each isoclinism class contains a special Lie superalgebra called a stem Lie superalgebra( its centre is contained in its derived subalgebra), as it is the case for Lie algebras (see in \cite{Moneyhun1994}).
\smallskip

Unlike the case of groups it has been shown that all covers of a finite dimensional Lie algebras are isomorphic \cite{Moneyhun1994}. Also in \cite{Salemkar2011} it is shown that each perfect Lie algebras, up to isomorphism has a unique cover. But for arbitrary Lie algebras it is an open problem whether the above result holds or not. However Salmekar et. al., \cite{Salemkar2008}, have shown covers are at least isoclinic, for any arbitrary Lie algebra with finite dimensional Schur multiplier. Here we show the same result holds for any arbitrary Lie superalgebra $L$, having finite dimensional multiplier.  

\smallskip

Let $X=X_{0} \cup X_{1}$ be a totally ordered $\mathbb{Z}_{2}$-graded set.  Let $\Gamma(X)$ be the groupoid of non-associative monomials in the alphabet $X$, $u \circ v=(u)(v)$ for $u, v \in \Gamma (X)$, and $S(X)$ be the free semigroup of associative words with the bracket removing homomorphism $-: \Gamma (X)\longrightarrow S(X) $. For $u=x_{1}...x_{n} \in S(X),~x_{i} \in X$, we consider the word length $l_{X}(u)=n$, the multidegree $m(u)$, $|u|=\sum^{n} _{i=1} |x_{i}| \in \mathbb{Z}_{2}$. Now let $K$ be a commutative ring with $1$, and let $A(X)$ and $F(X)$ be the free associative and non-associative $K$-algebras respectively. Let $A(X)_{\sigma}, F(X)_{\sigma}$  for $\sigma \in \mathbb{Z}_{2}$ be te $K$-linear spans of the subsets $S(X)_{\sigma}$ and $\Gamma(X)_{\sigma}$ respectively, $A(X)$ and $F(X)$ being the $\mathbb{Z}_{2}$-graded associative and non-associative algebras respectively. Let $I$ be the ideal generated by the homogeneous elements of the form $x\circ y - (-1)^{|x||y|} y\circ x$ and $(x\circ y)\circ z-x\circ (y \circ z)-(-1)^{|x||y|} y \circ (x\circ z)$, for $x, y \in \Gamma(X)$ then $L(X)=F(X)/I$, is the free Lie $K$-superalgebra(see \cite{YAM2000}).
 \smallskip

Suppose the set $S(X)$ is ordered lexicographically.
A monomial $u \in \Gamma (X)$ is said to be regular if either $u \in X$ or;
\begin{enumerate}
\item $u=u_{1}\circ u_{2}$ where $u_{1},~ u_{2}$ are regular monomials with $\overline{u_{1}}> \overline{u_{2}}$:
\item $u=(u_{1}\circ u_{2})\circ u_{3}$ with $\overline{u_{2}} \leq \overline{u_{3}}$. 
\end{enumerate}

A monomial $u \in \Gamma (X)$ is said to be s-regular if either $u$
 is a regular monomial or $u=(v)(v)$ with $v$ a regular monomial and $|v|=1$. Then the set of all images of the s-regular monomials form a basis of the free Lie superalgebra $L(X)$.  So the basis of homogeneous components either consists of regular words or it also contains squares of regular odd words. Then the dimension of $L(X)$ (the analogue of Witt's formula) is given in Corollary 2.8 in \cite{YAM2000}.  Further Petrogradsky in \cite{Pet2000, Pet2003} suggest a new formula without considering the two cases and found same value.
\smallskip

Let $X=X_{0} \cup X_{1}$, $X_{0}=\{x_{1},...,x_{m} \},~X_{1}=\{y_{m+1},...,y_{m+n} \}$ be a totally ordered $\mathbb{Z}_{2}$-graded set and $L$ be the free Lie superalgebra generated by $X$, $\mu (l)$ the M\"obius function, defined as $\mu(1)=1, \mu(l)=0$ if $l$ is divisible by an square$>1$, and $\mu(p_{1}\cdots p_{s})=(-1)^{s}$ if $p_{1}, \cdots, p_{s}$ are distinct primes. Suppose $L^{n}$ is the $n$-th term of the lower central series of $L$. 

\begin{theorem} [ \cite{Pet2000}\label{Cor1}, p. 221]
$\dim L^{c}= \frac{1}{c}\sum_{a\mid c} \mu(a)(m-(-1)^{a}n)^{\frac{c}{a}}$.
\end {theorem}

\begin{theorem}\label{theorem2}
Utilizing the above notations and assumptions the factor $F^{c}/F^{c+1}$ is an abelian Lie superalgebra of dimension $\frac{1}{c}\sum_{a\mid c} \mu(a)(m-(-1)^{a}n)^{\frac{c}{a}}$, and any element is of the form $[x_{i_{1}}, x_{i_{2}}, \cdots, x_{i_{c}}]$ where $x_{i_{1}}, x_{i_{2}}, \cdots, x_{i_{c}}  \in X$  is precisely a linear combination basic commutator of length $c$.
\end{theorem}

\smallskip

The organisation of the paper is as follows. In Section \ref{shl}, we give some chracterization of isoclisim of Lie superalgebras. Section \ref{mt}, is devoted to study structure of all covers of Lie superalgebras with some condition. In section \ref{st}, we give a necessary and sufficient condition for Lie superalgebras to be stem Lie superalgebras, and further show one more interesting result on isoclinism of Lie superalgebras. Also we prove a converse of Schur's theorem here. Finally in section \ref{ba} we have  generalised Schur's theorem and give a converse.

\section{Some properties of Isoclinism of Lie superalgebras}\label{shl}

Here we discuss some elementary results on isoclinism of Lie superalgebras.  
\begin{lemma}
If $L$ is a Lie superalgebra and $A$ be an abelian Lie superalgebra, then $L\sim L \oplus A$.
\end{lemma}
\begin{proof}
 $A$ is abelian, clearly we have $Z(L \oplus A)= Z(L)\oplus A$. Now
define the map
$\alpha: \frac{L}{Z(L)} \longrightarrow \frac{L\oplus A}{Z(L) \oplus A}$
by
 $x+Z(L) \mapsto x+(Z(L) \oplus A)$,
 for all $x \in L$.
It is easy to check that the map is well defined. Any homogenous element $x$ of the quotient Lie superalgebra $\frac{L}{Z(L)}$ is an element $x \in \frac{L_{\gamma}+Z(L)}{Z(L)}$ for $\gamma \in \mathbb{Z}_{2}$. Also any homogeneous element $y$ of $\frac{L\oplus A}{Z(L) \oplus A}$ is $y \in \frac{L_{\gamma}\oplus A_{\gamma}+Z(L) \oplus A}{Z(L) \oplus A}$ for $\gamma \in \mathbb{Z}_{2}$.
So, $\alpha$ is a homogeneous linear map of degree $0$ and clearly $\alpha$ is a bijection. Define map $\beta: L^{2} \longrightarrow (L\oplus A)^{2}=L^{2}$ is the identity map. Now, the diagram
\begin{center}
$\xymatrix{
\frac{L}{Z(L)} \times \frac{L}{Z(L)} \ar[d]^{\alpha \times \alpha} \ar[r]^{\quad\quad\phi} &L^{2}\ar[d]^{\beta}\\
\frac{L \oplus A}{Z(L)\oplus A} \times \frac{L \oplus A}{Z(L)\oplus A} \ar[r]^{\quad\quad\psi}          &L^{2}}$
\end{center}
is commutative as required.
\end{proof}
\begin{lemma}
Let $L$ be a Lie superalgebra and $H$ be a subalgebra of $L$. Then $H \sim H + Z(L)$. In particular, if $L = H + Z(L)$ then $L \sim H$. Conversely, if $L/Z(L)$ is of dimension $(m \mid n)$  and $L \sim H$ then $L=H+Z(L).$
\end{lemma}
\begin{proof}
Define the map
$\alpha: \frac{H}{Z(H)} \longrightarrow \frac{H+Z(L)}{Z(H)+ Z(L)}$ by
 $x+Z(H) \mapsto x+(Z(H)+ Z(L))$ for all $x \in L$ and $\beta: H^{2} \longrightarrow H^{2}$ is the identity map. Evidently, $\alpha$ and $\beta$ are Lie superalgebra isomorphisms such that the diagram
\begin{center}
$\xymatrix{
\frac{H}{Z(H)} \times \frac{H}{Z(H)} \ar[d]^{\alpha \times \alpha} \ar[r]^{\quad\quad\phi} &H^{2}\ar[d]^{\beta}\\
\frac{H+Z(L)}{Z(H)+Z(L)} \times \frac{H+Z(L)}{Z(H)+Z(L)} \ar[r]^{\quad\quad\psi}          &H^{2}}$
\end{center}
is commutative.
Hence $H \sim H+Z(L)$ and if $L=H+Z(L)$ then $H \sim L$.

\bigskip

Conversely suppose $K= H+Z(L)$ then $K \subseteq L$ is a subalgebra of $L$. Let $K$ be a proper subalgebra. Since $L \sim H$ and also $H \sim H+Z(L)=K$, so
$L \sim K$. Thus we have $\frac{L}{Z(L)} \cong\frac{K}{Z(K)}$ and hence the quotient $\frac{K}{Z(K)}$  is of dimension $(m \mid n)$. Since $Z(L) \subseteq Z(K)$, so $\dim\left(\frac{K}{Z(K)}\right)\leq \dim\left(\frac{K}{Z(L)}\right)$.  Say $\dim Z(L)=(r \mid s)$. We get
 $(m \mid n)\leq \dim(\frac{K}{Z(L)}) \leq \dim(\frac{L}{Z(L)})=(m \mid n)$ implies
$\dim\left(\frac{K}{Z(L)}\right)=(m \mid n)$. Therefore, $\dim K=(m+r | n+s)=\dim L$ which is a contradiction to our assumption that $K$ is proper. Hence $K = L$ as required.
\end{proof}

\begin{lemma}\label{lemma3}
Let $L$ be a Lie superalgebra and $I$ be a graded ideal. Then 
 $L/I \sim L/(I \cap L^{2})$. In particular, if $I \cap L^{2} =\{ 0\}$ then $L \sim L/I$. Conversely, if $L^{2}$ is finite dimensional and $L \sim L/I$ then $ I \cap L^{2}=\{0\}.$
\end{lemma}

\begin{proof}
 Let us denote $\bar{L}:= \frac{L}{I}$ and $\tilde{L} := \frac{L}{I\cap L^{2}}$. Define map $\alpha$ as follows,
\begin{align*} 
\alpha: \frac{\bar{L}}{Z(\bar{L})} &\longrightarrow \frac{\tilde{L}}{Z(\tilde{L})}\\
\bar{l}+Z(\bar{L})& \mapsto \tilde{l}+ Z(\tilde{L})
\end{align*}
where $\bar{l} \in\bar{L}$, $\tilde{l} \in \tilde{L}$ with $l \in L$.  At first we show  $\bar{l} \in Z(\bar{L})$ if and only if $\tilde{l}\in Z(\tilde{L})$. Consider, $\bar{l} \in Z(\bar{L})$ then we have $[\bar{l}, \bar{L}]=0$ i.e., $ [l, L] \in I$ and also $[l, L] \in L^{2}$ implies $[l, L] \in I \cap L^{2}$. Now for $x \in L$
\begin{align*}
[\tilde{l}, \tilde{L}] &= [l+L^2 \cap I, x+ I\cap L^2]\\
&= [l, x]+ I \cap L^2
= 0
\end{align*}
which shows $\tilde{l} \in Z(\tilde{L})$ and conversely. Let us take $\bar{l}+ Z(\bar{L})=\bar{m}+Z(\bar{L})$ where $\bar{l}, \bar{m} \in \bar{L}$. Then $(\bar{l}-\bar{m}) \in Z(\bar{L})$  which implies $(\tilde{l}-\tilde{m})  \in Z(\tilde{L})$. Thus  $\alpha$ is well defined. Also $\alpha$ is a Lie superalgebra isomorphism.
\smallskip

Define the map $\beta: \bar{L}^{2} \longrightarrow \tilde{L}^{2} $ as $\beta(\bar{l})=\tilde{l}$ where $\bar{l} \in \bar{L}^{2}, \tilde{l} \in \tilde{L}^{2}$ with $l \in L^{2}$. Clearly $\beta$ is well defined and a bijection. For $\bar{l}_{1},  \bar{l}_{2} \in \bar{L}^{2}$  we have $\beta([\bar{l}_{1}, \bar{l}_{2}]) = \beta([l_{1}, l_{2}] + I)= [l_{1} + (I\cap L^{2}), l_{2} + (I\cap L^{2})] =[\beta(\bar{l}_{1}), \beta(\bar{l}_{2})]$, shows that $\beta$ is an isomorphism. Finally we have for $\bar{l}, \bar{m} \in \bar{L}$ 
\begin{align*}
 \beta([\bar{l}, \bar{m}])
&= [\tilde{l}, \tilde{m}]\\
&= [\alpha(\bar{l}+Z(\bar{L})), \alpha(\bar{m}+Z(\bar{L}))],
\end{align*}
so the pair $(\alpha, \beta)$ is an isoclinism and $L/I \sim L/(I \cap L^{2})$.
 If $I \cap L^{2}=\{0\}$ then $L \sim \frac{L}{I}$.
 
 \bigskip
 
Conversely, let $\dim L^{2}=(m \mid n)$ and as $L \sim \frac{L}{I }$, we have the isomorphism, $L^{2} \cong \frac{L^{2}}{I \cap L^{2}}$. Hence $I \cap L^{2}=\{ 0\}$, as required.
\end{proof}

\begin{corollary} \label{corollary1}
Let $L$ and $K$ be two Lie superalgebras.
If $f: L \longrightarrow K$ is an onto homomorphism, then $f$ induces a isoclinism between $L$ and $K$  if and only if $ \ker f \cap L^{2}=\{0\}$. 
\end{corollary}

\begin{proof}
Consider the Lie superalgebra $L= L_{\bar{0}}\oplus L_{\bar{1}}$ and let $f$ be an onto homomorphism such that $\ker f \cap L^{2}=\{0\}$. Here $ \ker f= L_{\bar{0}} \cap \ker f\oplus L_{\bar{1}} \cap \ker f$ is a graded ideal of $L$. Using Lemma \ref{lemma3} we have $\frac{L}{\ker f} \sim L$.
 Again $\frac{L}{\ker f} \cong K$, so $L \sim K$. Conversely, if $f$ induces an isoclinism between $L$ and $K$, then $f \mid _{L^{2}}: L^{2}\longrightarrow K^{2}$ is an isomorphism. Hence, $\ker f \cap L^{2}= \{0\}$.
\end{proof}

\section{Covers of Lie superalgebras} \label{mt}

In this section at first we show that stem cover do exist for every Lie superalgebras. Further, we show that for any given Lie superalgebra with finite dimensional Schur multiplier, all of its covers are isoclinic.

\begin{lemma} \label{lemma4}
Stem covers exist for each Lie superalgebras.
\end{lemma}

\begin{proof}
Let $L = L_{\bar{0}} \oplus L_{\bar{1}}$ be a Lie superalgebra and let 
$\xymatrix{
0 \ar[r] & R\ar[r] & F \ar[r]^{\pi} & L \ar[r] & 0}$ be a free presentation of $L$. Clearly $\frac{R}{[R,F]}$ is a central graded ideal of $\frac{F}{[R,F]}$, \\ so
$\xymatrix{
0 \ar[r] & \frac{R}{[R, F]}\ar[r] & \frac{F}{[R, F]}\ar[r]^{\bar{\pi}} & L \ar[r] & 0}$ is a central extension of $L$. Let us consider $\frac{S}{[R, F]}$ be a complementary $\mathbb{Z}_{2}$-graded subspace of $\mathcal{M}(L)$ in $\frac{R}{[R, F]}$ for some graded ideal $S$ in $F$. Set $K := \frac{F}{S}$ and $M := \frac{R}{S}$ and hence $\frac{K}{M}= \frac{F/S}{R/S}\cong \frac{F}{R}$ implying $\frac{K}{M} \cong L$. Now
\begin{align*}
M= \frac{R}{S} \cong \frac{R/[R, F]}{S/[R, F]}
&= \frac{\mathcal{M}(L)\oplus S/[R, F]}{S/[R, F]}\\
& \cong \mathcal{M}(L).
\end{align*}
Again $M = \frac{R}{S} \subseteq \frac{F^{2}\cap R +S}{S} \subseteq \frac{F^{2}+S}{S}= \left(\frac{F}{S}\right)^{2}= K^{2}$ and also $M=\frac{R}{S} \subseteq Z\left(\frac{F}{S}\right)=Z(K)$, i.e. $M \subseteq Z(K) \cap K^{2}$. Hence, $0 \longrightarrow M \longrightarrow K \longrightarrow L \longrightarrow 0$ is the required stem cover we are looking for.
\end{proof}
The following lemma plays an essential role in our investigations.
\begin{lemma}\label{lem5}
Let $\xymatrix{
0 \ar[r] & R\ar[r] & F\ar[r]^{\pi} & L \ar[r] & 0}$ be a free presentation of a Lie superalgebra $L$ and let $\xymatrix{
0 \ar[r] & M\ar[r] & K\ar[r]^{\theta} & \tilde{L} \ar[r] & 0}$ be a central extension of another Lie superalgebra $\tilde{L}$. Then for each homomorphism $\alpha: L \longrightarrow \tilde{L}$, there exists a homomorphism $\beta: F/[R, F]\longrightarrow K$ such that $\beta(R/[R, F])\subseteq M$ and the following diagram is commutative:
\begin{equation}\label{eq3}
\xymatrix{
0 \ar[r] & \frac{R}{[R, F]}\ar[d]^{\beta_{1}}\ar[r] & \frac{F}{[R, F]}\ar[d]^{\beta}\ar[r]^{\bar{\pi}} & L\ar[d]^{\alpha} \ar[r] & 0\\
0 \ar[r] & M\ar[r] & K\ar[r]^{\theta} & \tilde{L} \ar[r] & 0.}
\end{equation}
\end{lemma}
\begin{proof} 
We have 
$\xymatrix{
0 \ar[r] & \frac{R}{[R, F]}\ar[r] & \frac{F}{[R, F]}\ar[r]^{\bar{\pi}} & L \ar[r] & 0}$ is a central extension of $L$.
As $F$ is free, there is a unique Lie superalgebra homomorphism $ F \longrightarrow \tilde{L}$. But since we have
\begin{center}
$\xymatrix{
F  \ar[r]^{\pi} &L\ar[d]^{\alpha}\\
K \ar[r]^{\theta}          &\tilde{L},}$
\end{center}
so there must exists an unique Lie superalgebra homomorphism $\beta': F \longrightarrow K$
 such that the following commutes:
\begin{center}
$\xymatrix{
F \ar[d]^{\beta'} \ar[r]^{\pi} &L\ar[d]^{\alpha}\\
K \ar[r]^{\theta}          &\tilde{L}}$
\end{center}
i.e., $\alpha \circ \pi= \theta \circ \beta^{'}$.
Our claim is $\beta'$ induces the homomorphism $\beta$. We have $R= \ker \pi$ is a graded ideal, now 
 let $x \in \beta'(R_{\gamma})$, so $x= \beta'(r)$ for $r \in R_{\gamma}$ for $\gamma \in \mathbb{Z}_{2}$. Since $M = \ker \theta$ is an graded ideal of $K$, take
$\theta(x)= \theta(\beta'(r))= \alpha(\pi(r))=0$ which implies $x \in M_{\gamma}$, i.e., $\beta'(R_{\gamma}) \subseteq M_{\gamma}$. Consider $[x, y]$ for $y \in F_{\delta} $ and $x \in R_{\gamma}$ with $\delta, \gamma \in \mathbb{Z}_2$, then $\beta'([x, y])= [\beta'(x), \beta'(y)]$. But we have $\beta'(x) \in M_{\gamma}$ and $M_{\gamma} \in Z(K)_{\gamma}$, implies 
$\beta'([x, y])=0$. We get $\beta'([R, F])=0 $, hence $\beta'$ induces the homomorphism $\beta$ as required. Also $\beta(\frac{R}{[R, F]}) \subseteq M $ implies $\beta_{1}= \beta|_{\frac{R}{[R, F]}}$ is the restriction map, hence a homomorphism such that the diagram in (\ref{eq3}) commutes.  
\end{proof}

 Following result is useful in determining structure of covers of Lie superalgebras having finite dimensional Schur multipliers.

\begin{theorem}\label{theorem1}
Let $L$ be a Lie superalgebra such that its Schur multiplier is of dimension $(m \mid n)$ and $\xymatrix{
0 \ar[r] & R\ar[r] & F\ar[r]^{\pi} & L \ar[r] & 0}$ be a free presentation of $L$. Then the extension $\xymatrix{
0 \ar[r] & M\ar[r] & K \ar[r]^{\psi} & L \ar[r] & 0}$ is a stem cover of $L$ if and only if there exists a graded ideal $S$ in $F$ such that:
\begin{enumerate}
\item $ K \cong \frac{F}{S}$ and $M \cong \frac{R}{S}$\\
\item $\frac{R}{[R, F]}= \mathcal{M}(L)\oplus \frac{S}{[R, F]}$.
\end{enumerate}
\end{theorem}

\begin{proof}
The same  proof as \cite[Theorem 2.2]{Salemkar2008} will work.
\end{proof}
Following are some important consequences of Theorem \ref{theorem1}.
\begin{corollary}
Let $L$ be a Lie superalgebra such that its Schur multiplier is of dimension $(m \mid n)$. Then all covers are isoclinic.
\end{corollary}
\begin{proof}
Let $0 \longrightarrow R \longrightarrow F \longrightarrow L \longrightarrow 0$ be a free presentation of $L$. We show that all covers of $L$ are isoclinc to a factor Lie superalgebra $\frac{F}{[R, F]}$. Let $0 \longrightarrow M \longrightarrow K \longrightarrow L \longrightarrow 0$ be a stem cover of $L$, i.e. the Lie superalgebra $K$ is any cover of $L$. By Theorem \ref{theorem1}, we have 
\begin{equation}
\frac{R}{[R, F]} = \mathcal{M}(L) \oplus \ker \beta.
\end{equation}
where $\beta: \frac{F}{[R, F]} \longrightarrow K$ is an onto homomorphism. Clearly $\mbox{ker} \beta \cap \left(\frac{F}{[R, F]}\right)^{2}= 0$ and hence using Corollary
 \ref{corollary1}, $K \sim \frac {F}{[R, F]}$.
\end{proof}
The following corollary give the existence of covers for finite dimensional Lie superalgebra.
\begin{corollary}
Any finite dimensional Lie superalgebra has at least one cover.
\end{corollary}
\begin{proof}
Let $\frac{F}{R}\cong L$ be a free presentation of Lie superalgebra $L$ and $\mathcal{M}(L) =\frac{F^2\cap R}{[R, F]}$. For some graded ideal $S$ of $F$, consider $\frac{S}{[R, F]}$ is complement of $\mathcal{M}(L)$ in $\frac{R}{[R, F]}$. Then by Theorem \ref{theorem1}, we get $0 \rightarrow \frac{R}{S}  \longrightarrow \frac{F}{S} \longrightarrow L \longrightarrow 0$ is a stem cover for $L$. Hence $\frac{F}{S}$ is the required cover of $L$.
\end{proof}
Stem extension of a finite dimensional Lie superalgebra is a homomorphic image of its stem cover. This is a result similar to work of Yamazaki \cite{Yamazaki1964} for group case and  to work by Salmekar et. al., \cite{Salemkar2008} for Lie algebra case.
\begin{theorem}\label{theorem2}
Let $0 \longrightarrow M \longrightarrow K \longrightarrow L \longrightarrow 0$ be a stem extension of finite dimensional Lie superalgebra L. Then there is a cover say $L^{*}$ of $L$ such that $K$ is a homomorphic image of $L^{*}$.
\end{theorem}
\begin{proof}
The same proof as \cite[Theorem 2.5]{Salemkar2008} will work.
\end{proof}
Following are some immediate consequences.
\begin{corollary}
The maximal stem extensions of Lie superalgebras are precisely the same as its stem covers. 
\end{corollary}
\begin{proof}
Let $0 \longrightarrow  M \longrightarrow K \longrightarrow L \longrightarrow 0$ be maximal stem extension of $L$. By Theorem \ref{theorem2} we have a stem cover $0 \longrightarrow M^{*} \longrightarrow L^{*} \longrightarrow L \longrightarrow 0 $ of $L$ such that $\phi:L^{*} \longrightarrow K$ is a homomorphism, hence necessarily an isomorphism.
\end{proof}
\begin{corollary}
Let $0 \longrightarrow M_{i} \longrightarrow K_{i} \longrightarrow L \longrightarrow 0$ here $i=1,2$,  be two maximal stem extensions of a Lie superalgebra $L$ of finite dimension, then $\dim K_{1}= \dim K_{2}=(m \mid n)$(say).
\end{corollary}
\section{structure of stem Lie superalgebra} \label{st}
Stem Lie algebras are first introduced and studied by Moneyhun \cite{Moneyhun1994}. Here we define stem Lie superalgebras and prove that each isoclinism family contains a stem Lie superalgebra.
\begin{definition}
A Lie superalgebra $L$ is called stem Lie superalgebra whenever $Z(L) \subseteq L^2$.
\end{definition}

\begin{lemma}\label{lem4}
Suppose  $\mathcal{C}$ is a isoclinic family of Lie superalgebras. Then
\begin{enumerate}
\item[$(1)$] $\mathcal{C}$ contains a stem Lie superalgebra.
\item[$(2)$] Each finite dimensional Lie superalgebra $T \in \mathcal{C}$ is stem if and only if $T$ has minimal even and odd dimension in $\mathcal{C}$.
\end{enumerate}
\end{lemma}
\begin{proof}
Let $L= L_{\bar{0}}\oplus L_{\bar{1}}$ be a Lie superalgebra and $ L \in \mathcal{C}$. So $Z(L) \cap L^2=L_{\bar{0}}\cap (Z(L)\cap L^2) \oplus L_{\bar{1}}\cap (Z(L)\cap L^2) $ is a graded ideal of $L$. Consider $S=L_{\bar{0}}\cap S \oplus L_{\bar{1}}\cap S$ be a $\mathbb{Z}_{2}$-graded vector space complement of $Z(L) \cap L^2$ in $Z(L)$, i.e., $Z(L)= Z(L)\cap L^2\oplus S$. Clearly $S$ is an graded ideal of $L$ and $S \cap L^2=0$. Denote $\frac{L}{S}=: T$ and using Lemma \ref{lemma3}
we have $T \sim L$. So $T \in \mathcal{C}$ and further
\begin{align*}
Z(T)=Z\left(\frac{L}{S}\right)
&=\frac{Z(L)+S}{S}\\
&= \frac{(Z(L) \cap L^2\oplus S)+S}{S}\\
&= \frac{S+Z(L)\cap L^2}{S}\\
&\subseteq \frac{S+L^2}{S}= \left(\frac{L}{S}\right)^{2}.
\end{align*} 
Hence $ T$ is a stem Lie superalgebra which proves $(1)$.

\bigskip

Consider $T \in \mathcal{C}$ and $T$ is stem Lie superalgebra with $\dim T=(m \mid n)$. Let $L \in \mathcal{C}$. Consider 
\begin{align*}
\frac{L^2}{L^2 \cap Z(L)} \cong \frac{L^2+Z(L)}{Z(L)}
&=\left(\frac{L}{Z(L)}\right)^{2}\\
&\cong \left(\frac{T}{Z(T)}\right)^{2}\\
&= \frac{T^{2}+Z(T)}{Z(T)}\\
& \cong \frac{T^{2}}{T^{2} \cap Z(T)}= \frac{T^{2}}{Z(T)}
\end{align*}
and also we have $L^2\cong T^2$. Hence $\dim Z(T)= \dim (Z(L) \cap L^2)=(r \mid s)$(say) with $\dim Z(T)_{\bar{0}}= \dim (Z(L) \cap L^2)_{\bar{0}}= r$ and $\dim Z(T)_{\bar{1}}= \dim (Z(L) \cap L^2)_{\bar{1}}$. Now $ \dim Z(T)= \dim (Z(L) \cap L^2)=(r \mid s) \leq \dim Z(L)=(p \mid q)$ where $r \leq p$ and $s \leq q$. As $\frac{L}{Z(L)} \cong \frac{T}{Z(T)}$ so $\dim \frac{L}{Z(L)} = \dim \frac{T}{Z(T)}=(m-r \mid n-s)$. Hence,\begin{align*}
\dim T &= (m-r \mid n-s)+\dim Z(T)\\
&\leq (m-r \mid n-s)+\dim Z(L) \\
&= (m-r+p \mid n-s+q)= \dim L,
\end{align*} i.e., $\dim T_{\bar{0}} \leq \dim L_{\bar{0}}$ and $\dim T_{\bar{1}} \leq \dim L_{\bar{1}}$, i.e. T is of minimum dimension as required.\\
 Conversely let $T \in \mathcal{C}$ and suppose $T$ has minimum even and odd dimension. Claim is $T$ is stem Lie superalgebra. Consider a complementary $\mathbb{Z}_{2}$-graded vector subspace $R$ of $Z(T)\cap T^2$ in $Z(T)$. Clearly $R \cap T^2=0$ and $R$ is an ideal in $T$. So $\frac{T}{R} \sim T$. But since $\dim T$ is minimum in $\mathcal{C}$, $\dim R=0$ implies $R=0$. We have $Z(T)= Z(T)\cap T^2$, i.e., $Z(T) \subseteq T^2$ as required.   
\end{proof}

A famous result of Schur is that if $G/Z(G)$ is finite then, so is the derived group. An analogue result of Schur for Lie superalgebra case is if $L/Z(L)$ is finite dimensional then so is $L^{2}$, which is also well known\cite{Nayak2019}. Further Niroomand \cite{Nir2010} proved a converse to Schur's theorem.  A generalisation of this is given by  Sury  \cite{Sury2010}. Here we extend the result of Sury to Lie superalgebra case as a converse of Schur's theorem for Lie superalgebras. For a finite set we denote order of $S$ as $\#S$.

\begin{theorem}\label{th5}
Let $L=L_{\bar{0}}\oplus L_{\bar{1}}$ be a Lie superalgebra such that its commutator set $S$ is finite. Then $L^2$ is finite. Further if $L/Z(L)$ is finitely generated by $(m \mid n)$ elements then $\dim L/Z(L)$ has finite dimension. More precisely, 
$$
\dim \left( L/Z(L) \right) \leq (m+n) \dim L^2.
$$ 
\end{theorem}
\begin{proof}
Let $S=\{[x_{i}, y_{i}] \mid 1 \leq i \leq t\}$. Consider a finitely generated subalgebra $H=\{x_{1}, \cdots, x_{k}, y_{1}, \cdots,y_{r}, x_{k+1}, \cdots, x_{t}, y_{r+1}, \cdots, y_{t} \}$ of $L$. Here $ x_{i}, y_{j} \in L_{\bar{0}}$ for $1 \leq i \leq k$ and $1 \leq j \leq r$, and $x_{k+1}, \cdots, x_{t}, y_{r+1}, \cdots y_{t} \in L_{\bar{1}}$. Clearly $S$ is the commutator set for $H$. Consider the natural projection $ \phi: H \longrightarrow H/Z(H)$. Let $H/Z(H)$ be generated by $ \phi(h_{1}), \cdots, \phi(h_{m}), \phi(h_{m+1}), \cdots, \phi(h_{m+n})$ where $h_{i} \in H_{\bar{0}}$ for $1 \leq i \leq m$ and $h_{m+j} \in H_{\bar{1}}$ for $1 \leq j \leq n$ . Denote $\phi(h_{i})=g_{i}, 1 \leq i \leq m+n$. Clearly centre of $H$ contains those $h \in H$ that commutes with $g_{i}$. Hence $Z(H)=\cap_{i=1}^{m+n} C_{H}(g_{i})$. Now, consider the adjoint representation of $H$, which is a homomorphism $\mbox{ad}: H \longrightarrow \mathfrak{gl}(H)$ defined as $\mbox{ad}_{g_{i}}(g_{j})=[g_{i}, g_{j}]$. For $g_{i} \in H$ kernel of $\mbox{ad}_{g_{i}}$ denoted as $\ker(\mbox{ad}_{g_{i}})$ is precisely $C_{H}(g_{i})$. Also note that $\dim \mbox{ad}_{g_{i}}(H) \leq \#S$. 

We have
\begin{align*}
\dim (H/Z(H)) &= \dim \left(\frac{H}{\cap_{i=1}^{m+n} \ker(\mbox{ad}_{g_{i}})} \right) \\
& \leq \dim \left(\sum_{i=1}^{m+n} \left(\frac{H}{\ker(\mbox{ad}_{g_{i}}}\right) \right)= \dim \left(\sum_{i=1}^{m+n} \mbox{ad}_{g_{i}}(H) \right)\\
&= (m+n) \#S.
\end{align*}
Now using Schur's theorem $H^{2}$ is finite dimensional. So $[L, L]=<S>=[H, H]$ is finite dimensional. Finally we have a Lie superalgebra $L$ with $L^{2}$ is finite dimensional and by the hypothesis $L/Z(L)$ is finitely generated by $(m \mid n)$ elements. Using the same arguments as earlier we have $\dim L/Z(L) \leq (m+n) \dim L^2$ where $L/Z(L)$ is generated by $(m\mid n)$ elements.
\end{proof}

Finally using the previous two results we have the following theorem.
\begin{theorem}
Let $L$ be a Lie superalgebra such that $L/Z(L)$ is finitely generated. If $\dim L^{2}=(r \mid s)$, then $L$ is isoclinic to a finite dimensional Lie superalgebra.
\end{theorem}
\begin{proof}
 Let $\mathcal{C}$ be an isoclinism class of Lie superalgebras with $L \in \mathcal{C}$. Using Lemma \ref{lem4}, $L$ is isoclinic to a stem Lie superalgebra say $T$. We have $Z(T) \subseteq T^{2}$ and $T ^{2}\cong L^{2}$, so $ \dim T^{2}=(r \mid s)$. It follows $ \dim Z(T) \leq r+s$. Further by Theorem \ref{th5}, $ L/Z(L)$ is finite dimensional. Suppose $\dim (L/Z(L)) = (m \mid n)$. This implies $\dim(T/Z(T)) =(m \mid n)$ as $T/Z(T) \cong L/Z(L)$. So, $T$ is finite dimensional, and $\dim T \leq m+r+n+s$.
\end{proof}

\section{ Baer's theorem and a converse}\label{ba}

In 1952 Baer \cite{Baer1952} extended the result of Schur and showed that  if $G/Z_{n}(G)$ is finite for a group $G$ then $G^{n+1}$ is finite too where $Z_{n}(G)$ and $G_{n+1}$ denote  the $n$th term of upper central series and $(n+1)$th term of lower central series respectively. The theorem is generalised by Salemkar et. al.\cite{Salemkar2009}. We prove that Baer's theorem still holds for Lie superalgebras.
\begin{theorem}\label{th6}
Let $L$ be a Lie superalgebra, and $\dim L/ Z_{c}(L)=(m \mid n)$ then \[\dim L^{c+1} \leq \frac{1}{c+1}\sum_{a\mid (c+1)} \mu(a)(m-(-1)^{a} n)^{\frac{c+1}{a}}.\]
\end{theorem}

\begin{proof}
Let $\{ x_{1}+Z_{c}(L), x_{2}+Z_{c}(L), \cdots, x_{m}+Z_{c}(L), x_{m+1}+Z_{c}(L), \cdots, x_{m+n}+Z_{c}(L) \}$ be a basis for $L/Z_{c}(L)$ where $x_{i} \in L_{\bar{0}}, 1 \leq i \leq  m$, and $x_{m+j} \in L_{\bar{1}}, 1 \leq j \leq n$. Any element of $L/ Z_{c}(L)$ can be written as $\sum_{i=1}^{m+n}  a_{i}x_{i} + h$ for $a_{i} \in \mathbb{F}$ and $h \in Z_{c}(L)$. Hence a generating set for $L^{c+1}$  precisely is to find all commutators of length $c+1$ from the list $\{x_{1}, x_{2}, \cdots, x_{m}, x_{m+1}, \cdots, x_{m+n}\}$. Hence $\dim L^{c+1} \leq \frac{1}{c+1}\sum_{a\mid (c+1)} \mu(a)(m-(-1)^{a} n)^{\frac{c+1}{a}}$.
\end{proof}
We can obtain the Schur's theorem as an immediate corollary to the Theorem \ref{th6}.
\begin{corollary}
Suppose $L$ is a Lie superalgebra with $\dim L/Z(L)=(m \mid n)$ then $\dim L^{2} \leq 1/2[(m+n)^{2}+(n-m)]$.
\end{corollary}
\begin{proof}
Clearly  $\dim L^{2} \leq \frac{1}{2}[\mu(1)(m+n)^{2}+\mu(2)(m-n)]=1/2[(m+n)^{2}+(n-m)]$.
\end{proof}
Now we state and prove a converse of Baer's theorem. Before we record here a lemma useful in proving the stated theorem.
\begin{lemma}\label{lemma1}
$\frac{Z_{j+i}(L)}{Z_{j}(L)}=Z_{i}\left(\frac{L}{Z_{j}(L)}\right).$
\end{lemma}
\begin{proof}
We induct on $i$. For $i=1$ we have definition of $j+1$-th centre of $L$. Assume result is true for $i$. For $i+1$, consider 
\begin{equation*}
\begin{split}
Z_{i+1}\left(\frac{L}{Z_{j}(L)}\right) &= \{ x+Z_{j}(L) \mid [x+Z_{j}(L), y+Z_{j}(L)] \in \frac{Z_{j+i}(L)}{Z_{j}(L)}, \forall y \in L \} \\
&=\{x+Z_{j}(L) \mid  [x, y] \in Z_{i+j}(L),  \forall y \in L \}\\
&= Z_{i+j+1}\left(\frac{L}{Z_{j}(L)}\right).
\end{split}
\end{equation*} First equality follows using induction hypothesis.
\end{proof}

\begin{theorem}
Let $L_{\bar{0}} \oplus L_{\bar{}1}$ be a Lie superalgebra such that the commutator set $S$ having elements of length $c+1$ is finite. Then $L^{c+1}$ is finite dimensional. Further, if $L/Z_{c}(L)$ is generated by $(m \mid n)$ elements then \[\dim L/Z_{c}(L) \leq (m+n)^{c} \dim L^{c+1}.\]
\end{theorem}
\begin{proof}
Let $S=\{s_{1}, \cdots s_{t}, s_{t+1}, \cdots, s_{t+k}\}$ be finite where $s_{i}$ are commutators of length $c+1$. Precisely $s_{i}=[x_{i,1}, x_{i,2}, \cdots, x_{i,c+1}]$ for $i=1, \cdots, t+k$. Let $H=\{x_{i,1}, x_{i,1}, \cdots, x_{i,c+1} |1\leq i \leq t+k \}$ where $ x_{i,r}$ for $1\leq r \leq c+1$ are homogeneous elements be a finitely generated subalgebra of $L$ such that $S$ is its commutator set having elements of length $c+1$.  We induct on $c$. If $c=1$ then $S$ is a finite commutator set of length $2$ and by Theorem \ref{th5}, $L^2$ is finite dimensional and $\dim L/Z_{2}(L) \leq (m+n) \dim L^2.$ 
\smallskip

Let us assume the result is true for $c$. With $S$ is finite commutator set of length $c$ then it is easy to see $L^{c}$ is finite dimensional. In addition if $L/Z_{c-1}(L)$ is finitely generated then from induction hypothesis \begin{equation}
\dim L/Z_{c-1}(L) \leq (m+n)^{c-1} \dim L^c.
\end{equation}
Our claim is the result is true for $c+1$. Let $L/Z_{c}(L)$ be generated by $g_{1}+Z_{c}(L), \cdots, g_{m}+Z_{c}(L), g_{m+1}+Z_{c}(L), \cdots , g_{m+n}+Z_{c}(L)$ where $g_{i} \in L_{\bar{0}}$ for $1\leq i \leq m$ and $g_{m+j} \in L_{\bar{1}}$
 for $1\leq j \leq n$. Clearly $Z_{c}(L)$ consists of those elements of $L^{c}$ which commutes with $g_{i}$. Consider the adjoint representation  $\mbox{ad}: L^{c} \longrightarrow \mathfrak{gl}(L)$, we have 
\[Z_{c}(L)=\cap_{i=1}^{m+n}C_{L^{c}}(g_{i})\ \;\; \mbox{and}\;\; C_{L^{c}}(g_{i})= \ker (\mbox{ad}_{L^{c}}).\]

 As $L/Z_{c-1}(L)$ is finitely generated so is  $\frac{L/Z(L)}{Z_{c-1}(L/Z(L))}$. Now 
\[ \left( \frac{L}{Z(L)}\right)^{c}=\frac{L^{c}+Z(L)}{Z(L)} \cong \frac{L^{c}}{L^c \cap Z(L)}\] is finite dimensional as $L^{c}$ is finite dimensional. By using Lemma \ref{lemma1}  we have $L/Z_{c}(L)= \frac{L/Z(L)}{Z_{c-1}(L/Z(L))}$. Hence using induction hypothesis,
\begin{equation}\label{eq5.2}
\dim \left(\frac{L}{Z_{c}(L)}\right) \leq (m+n)^{c-1} \dim \left(\frac{L}{Z(L)}\right)^{c}.
\end{equation}
By Theorem \ref{th6}, $L^{c+1}$ is finite dimensional.
 Consider
\begin{align*}
 \left(\frac{L}{Z(L)}\right)^{c} \cong \frac{L^{c}}{L^{c} \cap Z(L)} 
 &\cong \frac{L^{c}}{Z_{c}(L)}\\
  &=\left(\frac{L^{c}}{\cap_{i=1}^{m+n} C_{L^{c}}(g_{i})}\right) \\
&\cong \left(\sum_{i=1}^{m+n}\left(\frac{L^{c}}{\ker(\text{ad}_{L^{c}}}\right)\right)\\
&= \left(\sum_{i=1}^{m+n} \text{ad}_{L^{c}}(g_{i})\right).
\end{align*}
Hence $\dim  \left(\frac{L}{Z(L)}\right)^{c} \leq (m+n) \dim L^{c+1}.$
Now substituting this in \eqref{eq5.2}, we have the required result as follows,

\[\dim \left(\frac{L}{Z_{c}(L)}\right) \leq (m+n)^{c}  \dim L^{c+1}.\] 
\end{proof}

{\bf Acknowledgement.}  The author is supported by NBHM Post Doctoral Fellowship, Govt. of India. The author is grateful to the referees for the helpful corrections.

\medskip

\end{document}